\documentclass[12pt]{amsart}
\usepackage{amsmath}
\usepackage{graphicx}
\usepackage{amssymb,amsthm}
\usepackage{hyperref}
\usepackage{esint}
\usepackage{epstopdf}
\usepackage{color}
\usepackage{url}
\usepackage[top=1in, left=1in, right=1in, bottom=1in]{geometry}

\newtheorem{theorem}{Theorem}[section]
\newtheorem{lemma}[theorem]{Lemma}
\newtheorem{proposition}[theorem]{Proposition}
\newtheorem{corollary}[theorem]{Corollary}

\newtheorem{remark}[theorem]{Remark}

\def\R{{\mathbb R}}

\def\a{\alpha}
\def\b{\beta}
\def\e{\varepsilon}

\def\l{\lambda}

\def\m{\mu}
\def\n{\nabla}
\def\p{\partial}
\def\r{\rho}

\def\t{\tau}

\def\w{\omega}
\def\W{\Omega}

\def\1{\left(}
\def\2{\right)}
\def\3{\left\{}
\def\4{\right\}}
\def\8{\infty}

\def\ss{\subseteq}
\def\cc{\subset\subset}

\DeclareMathOperator*{\supp}{supp}
\DeclareMathOperator*{\osc}{osc}

\begin{document}

\title[Optimal regularity for obstacle-like problem with logarithm]{Optimal regularity for a two-phase obstacle-like problem with logarithmic singularity}

\author{Dennis Kriventsov}
\address[Dennis Kriventsov]{Department of Mathematics, Rutgers University, 110 Frelinghuysen Road, Piscataway, NJ}
\email{dnk34@math.rutgers.edu}

\author{Henrik Shahgholian}
\address[Henrik Shahgholian]{Department of Mathematics, KTH Royal Institute of Technology, Stockholm, Sweden}
\email{henriksh@kth.se}  

\date{September 6, 2020}

\begin{abstract} 
We consider the semilinear problem 
\[
\Delta u =  \lambda_+ \left(-\log u^+\right) 1_{\{u >  0\}}  -  \lambda_- \left(-\log u^- \right) 1_{\{u <  0\}}
\qquad  \hbox{ in } B_1,
\]
where $B_1$ is the unit ball in $\R^n$ and assume $\lambda_+, \lambda_- > 0$. Using a monotonicity formula argument, we prove an optimal regularity result for solutions: $\n u$ is a log-Lipschitz function.

This problem introduces  two main difficulties. The first is  the lack of invariance in the scaling and blow-up of the problem.  The other (more serious) issue is a term in the Weiss energy which is potentially non-integrable unless one already knows the optimal regularity of the solution: this puts us in a catch-22 situation.

\end{abstract}

\maketitle

\section{Introduction and Main results}

In \cite{QS} the authors study the semilinear elliptic problem
\begin{equation}\label{eq:main1} 
\Delta u = \left(-\log u\right) 1_{\{u >  0\}} \quad  \hbox{and} \quad  u \geq 0 \qquad  \hbox{ in } B_1,
\end{equation}
proving optimal regularity as well as non-degeneracy of the solutions: essentially, if $d = d(x, \p \{u - 0\})$, solutions behave like $d^2 |\log d|$. From here one could also show that the free boundary has zero Lebesgue measure. 

It should be remarked that the main points of interest are along  $\partial \{ u > 0\}$. We assume that the origin is such a point, and will discuss the free boundary around the origin.  Since such an analysis is local and disregards the behavior of solutions far away from the free boundary $\partial \{ u > 0\}$, we omit the boundary values on $\partial B_1$. Also all statements about regularity will be uniform in the half-ball $B_{1/2}$, where the norms depend on some norm of the solution in the unit ball and dimension.

One of the key difficulties in this problem becomes apparent when one studies the blow-up limits of such solutions, $\frac{u(r \cdot)}{ 2r^2 |\log r|}$: it is not hard to verify (see Section \ref{sec:prelim}) that these converge locally uniformly, along subsequences, to entire solutions of the classical obstacle problem
\[
	\Delta u =  1_{\{u >  0\}} \quad  \hbox{and} \quad  u \geq 0 \qquad  \hbox{ in } \R^n.
\]
On the one hand, this problem is well-understood (see \cite{PSU} for a general reference); on the other hand, the scaling for this problem is different from the scaling of \eqref{eq:main1}. This lack of invariance makes it difficult to gain information from compactness arguments in this setting. The other major difficulty is more apparent when one differentiates the equation: $\n u$ solves $\Delta \n u = - \frac{\n u}{u} 1_{\{u > 0\}}$, which is extremely singular near the origin. As most methods for studying the regularity of the free boundary involve differentiating the equation in this fashion (directly or indirectly), they are difficult to apply here. In fact, the techniques employed in \cite{QS} do not seem sufficient to prove regularity of the free boundary or otherwise go beyond what is shown there.

In this short note our goal is to explore some alternative tools available for studying problems like \eqref{eq:main1}. We consider a two-phase version of it here:
\begin{equation}\label{eq:main2} 
\Delta u =  -\lambda_+ \left(\log u^+ \right) 1_{\{u >  0\}}  + \lambda_- \left(\log u^- \right) 1_{\{u <  0\}}
     \qquad  \hbox{ in } B_1,
\end{equation}
where $u^\pm = \max (\pm u, 0)$. Unlike for \eqref{eq:main1}, the optimal regularity of solutions to this is not known in the literature, and does not follow from the arguments in \cite{QS}. Our main result here proves it using a rather different technique:
\begin{theorem}\label{thm:main1}
Let $u$ be a solution to \eqref{eq:main2}.   Then  $u \in C^{2 -\log} (B_{1/2}) (0)$, i.e.
\[
|\nabla u (x) -\nabla u (y) | \leq  C|x -y|  |\log |x - y||,
\]
for a constant $C(n, \l_+, \l_-, \int_{B_1}u^2)$.
\end{theorem}
The proof follows an approach to regularity using a Weiss-type monotonicity formula, somewhat like in \cite{ASUW}. The major difference is that it is not actually clear that the Weiss energy associated to this problem is almost monotone \emph{a priori}; unless one already knows the optimal regularity of solutions, there is potentially a non-integrable error term in the monotonicity formula. Our argument, therefore, is simultaneously proving the monotonicity of the appropriate energy and the regularity of $u$, not just using the former to establish the latter. Moreover, to make this work we also have to use a nearly-optimal regularity result for $u$, Lemma \ref{lem:log2reg}, within the argument to estimate some of the problematic terms in the monotone quantity.

While we plan to study the regularity of the free boundary for \eqref{eq:main1}, \eqref{eq:main2} and related questions in future work, this seems to be more delicate and is not treated here.

The structure of this paper is as follows: Section \ref{sec:prelim} establishes notation and collects useful results readily available in the literature. Then Section \ref{sec:subopt} covers straightforward suboptimal regularity results which are nonetheless needed in later sections. In Section \ref{sec:growth} we prove the key growth lemma, which contains the main ideas of this note. Finally, Section \ref{sec:proofmain} gives a proof of Theorem \ref{thm:main1} using this growth lemma and some PDE techniques, while Section \ref{sec:nondegeneracy} presents a nondegeneracy property that shows that Theorem \ref{thm:main1} is essentially optimal.

\section{Preliminary Analysis} \label{sec:prelim}

\subsection{Definitions and Notation}

Let $\W$ be a smooth domain, and consider minimizers to the functional
\[
E(u; \W) = \int_{\W} \frac{1}{2}|\n u|^2 + F(u),
\]
where
\[
F(t) = \begin{cases}
(\l_+ t_+ + \l_- t_-) (1 - \log |t|) & t\neq 0 \\
0 & t = 0 .
\end{cases}
\]
In particular, we say that  \emph{$u$ is a minimizer of $E$ on $\W$} if $E(u; \W) \leq E(v; \W)$ among all functions $v$ in $H^1(\W)$ with $v - u \in H^1_0(\W)$. Given a function $v \in H^1(\W)$, it is not difficult to see that $E$ admits a minimizer $u$ to $E$ with $u - v\in H^1_0(\W)$ using the direct method, but it is not clear whether or not $u$ is a unique minimizer. We do not address this question of uniqueness here, but our results apply to any minimizer. We will use the notation $E(u)$ for $E$ where the choice of domain $\W$ is clear from the context. Since our analysis is mainly local we will generally assume $\Omega = B_1(0)$.

Minimizers of $E$ will be solutions of \eqref{eq:main2} on $B_1$, in the weak sense. They will also be analytic functions on $\{u \neq 0\}$. Note that, on the other hand, it is not clear that every solution to $\eqref{eq:main2}$ is a minimizer of $E$, as (unlike with the classical obstacle problem) the functional $E$ is not convex. We will deal only with minimizers in this paper.

Let us define the following rescaled functions and rescaled $E$, for $r<1$:
\[
u_r(x) = \frac{u(r x)}{r^2(1 - 2\log r)}
\]
and
\[
E_r(v; \W) = \int_{\W} \frac{1}{2} |\n v|^2 + F_r(v),
\]
where
\[
F_r(t) = \begin{cases}
(\l_+ t_+ + \l_- t_-) (1 - \frac{\log |t|}{1 - 2 \log r} - \frac{\log (1 - 2\log r)}{1 - 2\log r}) & t \neq 0\\
0 & t = 0.
\end{cases}
\]
These have the following property: if $u$ minimizes $E$ on $B_1$, then $u_r$ minimizes $E_r$ on $B_1$. Note that at least when evaluated on a smooth, fixed $v$, the last two terms of $E_r$ tend to zero as $r\rightarrow 0$, leading to
\[
E_0(v; \W) = \int_{\W} \frac{1}{2} |\n v|^2 + F_0(v). 
\]
where
\[
F_0(t) = \l_+ t_+ + \l_- t_-.
\]
This, however, is a convex functional whose minimizers coincide with solutions to the two-phase obstacle problem
\begin{equation}\label{eq:2phaseobst}
\Delta u = \l_+ 1_{\{u>0\}} - \l_- 1_{\{u < 0\}}.
\end{equation}

\subsection{Basic Regularity of Solutions}

In \cite{GG}, the authors show (a much more general version of) the following:
\begin{proposition} \label{prop:C1a}
	Let $u$ be a minimizer of
	\[
	G(u) = \int_{B_1} \frac{1}{2}|\n u|^2 + g(u),
	\]
	where $|g(t) - g(s)|\leq C_0 |t - s|^{\a_0}$. Then
	\[
	\|u\|_{C^{1,\a_1}(B_{1/2})} \leq C(n, C_0, \a_0, \|u\|_{L^2(B_1)}),
	\]
	where $\a_1$ depends only on $\a_0$.
\end{proposition}

Note that the function $F$ does not actually satisfy the assumptions here in general: while $F(t)$ is locally H\"older continuous, it grows like $|t||\log |t||$ for large $|t|$ and so does not admit a uniform modulus. It does, however, satisfy the assumptions so long as $\sup_{B_1} |u|$ is bounded, and so the proposition may be applied with extra dependence on $\sup_{B_1}|u|$. The following proposition gives an estimate on this quantity, based only on the following fact about $F_r$:
\[
|F_r(t)| \leq C |t| (1 + |\log |t|| ) \leq C_*(1 + t^2).
\]
Using this, we may apply the results of Section 2 and 3 of \cite{GG2} to obtain:
\begin{proposition}\label{prop:degiorgi}
	Let $u$ be a minimizer of $G$ on $B_1$, with $|g(u)| \leq C_*(1 + t^2)$. Then there is a constant $\a$ depending only on $n$ such that
	\[
	\|u\|_{C^{0, \a}(B_{1/2})} \leq C(C_*, n, \|u\|_{L^2(B_1)}).
	\]
\end{proposition}
The proof there is based on directly verifying a Cacciopoli inequality and then applying De Giorgi's technique.

A straightforward consequence of the above propositions is the following lemma:
\begin{lemma}\label{lem:conv}
	Let $u_k$ be minimizers of $G_{k}(u) = \int \frac{1}{2}|\n u|^2 +g_k(u)$ on $B_1$ with 
	\[
		\sup_k \|u_k\|_{L^2(B_1)} \leq C < \8.
	\]
	 Assume that $g_k$ satisfy $|g_k(t) - g_k(s)|\leq C |t - s|^{\a_0}$ uniformly in $k$, and converge to $g(t)$ locally uniformly. Then, along a subsequence, $u_k \rightarrow u$ on $B_{1/2}$ in $C^{1,\a}$ topology for some $\a>0$, and $u$ is a minimizer of $G(u) = \int \frac{1}{2}|\n u|^2 +g(u)$.
	
	In particular, if $g_k = F_{r_k}$ with $r_k \searrow 0$, and $|u|\leq C$, the assumption on $g_k$ is verified and $u$ solves \eqref{eq:2phaseobst}.
\end{lemma}

\begin{proof}
	Applying Proposition \ref{prop:C1a}, we have that
	\[
	\|u_k\|_{C^{1,\a_1}(B_{1/2})} \leq C
	\]
	uniformly in $r$. This immediately gives the convergence along a subsequence to a function $u$.
	
	Next, we note that to show that $u$ minimizes $G$, it suffices to check that for any $v \in C^{\8}(B_{1/2})$ and $\supp(v-u) \cc B_{1/2}$, we have that $G(v; B_{1/2}) \geq G(u; B_{1/2})$. To that end, let $\eta$ be a smooth cutoff which is equal to one on $B_{1/2 - \r}$ and vanishes on $\p B_{1/2}$. Set $v_k = \eta v + (1-\eta) u_k$; this is a valid competitor for $G_k$, so
	\[
	G_k(v_k; B_{1/2}) \geq G_k(u_k; B_{1/2}).
	\]
	Set $v_\8 = \eta v + (1-\eta) u$.

	Now, as $\n u_k \rightarrow \n u$, $g_k(u_k) \rightarrow g(u)$, and $g_k(v_k) \rightarrow g(v_\8)$ uniformly, we have that
	\[
	G(v_\8; B_{1/2}) \geq G(u; B_{1/2}).
	\]
	Choosing $\r$ small enough, we see that $v_\8 = v$, and this implies the conclusion.
	
	Finally, observe that the integrands $F_r$ satisfy $|F_r(t) - F_r(s)| \leq C_0 |t - s|^{\a_0}$ for any $\a_0 <1$ and $C_0$ independent of $r$ if $t, s < C/{r^2(1 - 2 \log r)}$, and converge uniformly to $F_0$ as $r\rightarrow 0$. The last conclusion follows by noting that \eqref{eq:2phaseobst} is the Euler-Lagrange equation for an $E_0$ minimizer. 
\end{proof}

Recall that minimizers of $G$ need not be unique (except in the case of $G = F_{0}$, which is convex). As such, this lemma should not be thought of as a stability property for minimizers of $G$ but rather a closure property for minimizing families. Our intended use for it is in compactness and blow-up arguments.

\section{Suboptimal Regularity} \label{sec:subopt}

The function $F_r$ is continuous, and in fact satisfies
\[
|F_r(t) - F_r(s)| \leq C |t - s| (1 + |\log |t - s||)
\]
uniformly in $r$. Indeed, when $|t - s| > |s|/2$, this follows from
\begin{align*}
|F_r(t) - F_r(s)| &\leq |F_r(t)| + |F_r(s)| \\
 &\leq C[|t| (1 + |\log |t||) + |s| (1 + |\log |s||)] \\
 &\leq C |t - s| (1 + |\log |t - s||),
\end{align*}
where the last inequality used that $|t| \leq |s| + |t - s| < 3 |t - s|$. On the other hand, if $|t -s | \leq \min\{|t|, |s|\}/2$, this implies that $t$ and $s$ have the same sign and
\begin{align*}
|F_r(t) - F_r(s)| & \leq |t-s|\max_{z\in [s,t]}|F'_r(z)|\\
& \leq C |t - s| (1 + \max\{|\log |s||, |\log |t||\})\\
& \leq C |t - s| (1 + |\log |t - s||).
\end{align*}

Setting
\[
\eta(t) = C|t|(1 + |\log |t||)
\]
for the remainder of this section, we can obtain an optimal regularity estimate for minimizers of
\[
G(u) = \int \frac{1}{2}|\n u|^2 +g(u),
\]
where 
\begin{equation}\label{eq:loglip}
|g(t) - g(s)| \leq \eta(t -s)
\end{equation}
This is not the optimal regularity for minimizers of $F$ and $F_r$, which will be discussed in the next section, but surprisingly we will require it anyway. The proof is a simple application of the methods of \cite{GG}.

The only properties of $\eta$ which are needed below are:
\begin{enumerate}
	\item $\eta :[0, 1] \rightarrow [0, C]$ is an increasing continuous bijection.
	\item $\eta$ is concave.
	\item $t \mapsto \frac{\eta(t)}{t}$ is a nonincreasing function.
	\item $ \eta(t) \leq C t^{\frac{1}{2} + \a}$ for some $C, \a$ and all $t \in [0, 1]$.
\end{enumerate}
Note that (1) and (3) imply that $c t \leq \eta(t)$ for $t \leq 1$.

\begin{lemma} \label{lem:harmonicapprox}
	Let $u$ be a minimizer of $G$ on $B_r$, $r\leq 1$, with $G$ satisfying \eqref{eq:loglip}. There is a $c_*(\eta)$ such that if $\osc_{B_r} u\leq c_*$ and $h$ is a harmonic function with the same boundary values as $u$ along $\p B_r$, then
	\[
	\fint_{B_r}|\n (u - h)|^2 \leq C \frac{\eta^2(r^2)}{r^2}.
	\]
\end{lemma}

\begin{proof}
	As $h$ is harmonic, $\int_{B_r} \n h \cdot \n(u - h) = 0$, so
	\[
	\int_{B_r} |\n (u-h)|^2 = \int_{B_r} \n (u + h) \cdot \n (u - h) = \int_{B_r} |\n u|^2  - |\n h|^2.
	\]
	Using $h$ as a competitor for $u$ in the minimization of $G$ gives $G(u) \leq G(h)$, and so
	\begin{equation}\label{eq:harmonicapprox1}
	\int_{B_r} |\n u|^2  - |\n h|^2 \leq 2 \int_{B_r} g(h) - g(u) \leq 2\int_{B_r}\eta(h -u).
	\end{equation}
	From the assumption on the oscillation of $u$ and the maximum principle, $|u-h|\leq 2c_* < 1$, while the modulus $\eta$ is concave. This can be used to show that
	\begin{equation}\label{eq:harmonicapprox2}
	\frac{1}{|B_r|}\int_{B_r}\eta(h -u) \leq \eta(\frac{1}{|B_r|}\int_{B_r}|h -u|) \leq \eta((\frac{1}{|B_r|} \int_{B_r} |h - u|^{\frac{2n}{n-2}})^{\frac{n-2}{2n}}) \leq \eta(2c_*)
	\end{equation}
	from Jensen's inequality. Applying the Sobolev embedding, we get
	\[
		\int_{B_r}\eta(h -u) \leq |B_r| \eta( C |B_r|^{\frac{2 - n}{2n}} (\int |\n (u - h)|^2)^{\frac{1}{2}} ).
	\]
	Rewriting and plugging into \eqref{eq:harmonicapprox1} gives	
	\[
	\frac{1}{|B_r|}\int_{B_r} |\n (u-h)|^2 \leq 2\eta( C r (\frac{1}{|B_r|} \int_{B_r} |\n (u -h)|^2)^{1/2}).
	\]

	Setting $A = \frac{1}{|B_r|}\int_{B_r} |\n (u-h)|^2$, we have shown that
	\[
	A \leq 2 \eta( C r \sqrt{A}).
	\]
	We also have $A \leq 2\eta(2c_*) < \frac{1}{C}$ by inserting \eqref{eq:harmonicapprox2} into \eqref{eq:harmonicapprox1} directly and taking $c_*$ small enough.

	Now, if $A \geq r^2/C^2$, then
	\[
		\frac{\eta(C r \sqrt{A})}{C r \sqrt{A}} \leq \frac{\eta(r^2)}{r^2},
	\]
	giving $ A \leq 2 C r \sqrt{A} \frac{\eta(r^2)}{r^2}$, and so $ A \leq 4 C^2 \frac{\eta^2(r^2)}{r^2}$. This means that
	\[
	A \leq \max\{4 C^2 \frac{\eta^2(r^2)}{r^2}, r^2/C^2\} \leq C' \frac{\eta^2(r^2)}{r^2}.
	\]
\end{proof}

\begin{lemma}\label{lem:harmonic} Let $h$ be a harmonic function with finite Dirichlet energy on $B_r$ and with $h(0) = 0$, and let $s<r$. Then
	\[
	\int_{B_s}|h|^2 \leq \1 \frac{s}{r}\2^{n+2} \int_{B_r}|h|^2.
	\]
\end{lemma}

\begin{proof}
	Write $h = \sum_{i = 1}^\8 \a_i P_i$, where each $P_i$ is a harmonic polynomial and the $P_i$ are orthonormal in $L^2(B_1)$. By the assumptions made, each $P_k$ is of degree at least $1$. Hence
	\[
	\int_{B_t}|h|^2 = \sum_{i = 1}^\8 \a_i^2 t^{n + 2 \deg P_k};
	\]
	after dividing by $t^{n+2}$ we see that the right hand side is a nondecreasing function. This gives the conclusion.
\end{proof}	

\begin{lemma} \label{lem:log2decay}
	Let $u$ be a minimizer of $G$ on $B_r$, $r\leq 1$, with $G$ satisfying \eqref{eq:loglip}. Assume that $\osc_{B_r} u\leq c_*$. Then there is a constant $C = C(n, \eta) > 0$ such that
	\[
	\frac{1}{\r^{n-2} \eta^2(\r^2)}\int_{B_\r}|\n u - \fint_{B_\r}\n u|^2 \leq C\1 1 + |\log \r/r|^2 + \frac{1}{r^{n-2} \eta^2(r^2)}\int_{B_r}|\n u - \fint_{B_r}\n u|^2\2
	\]
	for any $\rho < r$.
\end{lemma}

\begin{proof}
	First, fix $\r_1 < \r_2 \leq r$ and let $h$ be the harmonic function on $B_{\r_2}$ which coincides with $u$ on $\p B_{\r_2}$. Then, from Lemma \ref{lem:harmonic} applied to the components of $\n h - \n h(0)$ (noting that $\fint_{B_{s}}\n h = \n h(0)$), we have
	\[
	\int_{B_{\r_1}}|\n h - \fint_{B_{\r_1}}\n h|^2 \leq \1\frac{\r_1}{\r_2}\2^{n+2} \int_{B_{\r_2}}|\n h - \fint_{B_{\r_2}}\n h|^2.
	\]
	Applying Lemma \ref{lem:harmonicapprox} to $u$ and $h$ on $B_{\r_2}$, this gives
	\[
	\1\int_{B_{\r_1}}|\n u - \fint_{B_{\r_1}}\n u|^2\2^{1/2} \leq \1\1\frac{\r_1}{\r_2}\2^{n+2} \int_{B_{\r_2}}|\n u - \fint_{B_{\r_2}}\n u|^2\2^{1/2} + C \r_2^{n/2-1} \eta(\r_2^2).
	\]

	Setting $\phi(t) = \1\frac{1}{t^{n-2} \eta^2(t^2)} \int_{B_{t}}|\n u - \fint_{B_{t}}\n u|^2\2^{1/2}$, we have shown that for any $t < s \leq r$,
	\[
	\phi(t) \leq \1\frac{t^2\eta(s^2)}{s^2\eta(t^2)}\2\phi(s) + C \frac{s^{n/2-1}\eta(s^2)}{t^{n/2-1}\eta(t^2)} \leq \phi(s) + C \frac{s^{n/2+1}}{t^{n/2+1}}.
	\]	
	Setting $t = \t s$ in the above gives
	\[
	\phi(\t s) \leq \phi(s) + C \t^{-n/2-1}.
	\]
	Iterating,
	\begin{align*}
	\phi(\t^{k+1} s) &\leq \phi(\t^{k}s) + C \t^{-n/2-1}\\
	&\leq \phi(s) + C (k+1) \t^{-n/2-1}\\
	&\leq \phi(s) + C \frac{|\log(t/s)|}{|\log \t|} \t^{-n/2-1}
	\end{align*}
	where $t = \t^{k + 1}s$. Now fix $\t < 1$ (for example, $\t = 1/2$), set $s = r$, and choose $k$ such that $\r \in [\t^{k+1}s, \t^k s]$, to obtain
	\[
	\phi(\r) \leq \phi(\t^k r) + C \t^{-n/2-1} \leq \phi(r) + C(\t)(1 + |\log \r/r| ).
	\]
\end{proof}

Below set
\[
\eta_1^2(t) = \int_0^t \frac{\eta^2(s^2)}{s^2}(1 + |\log s|^2) ds.
\]
From assumption (4) on $\eta$, this is a finite increasing function of $t$. We will use that
\[
\sum_{k = K}^\infty \frac{\eta^2(2^{-2k}r_0^2)}{2^{-2k}r_0^2} (1 + k^2) \approx \eta_1^2(2^{-K}r_0)
\]
and the doubling property
\[
\eta_1(2r) \leq C\eta_1(r)
\]
below. For $\eta(t) = t(1 + |\log t|^J)$, one may compute $\eta_1(t) \approx t (1 + |\log t|^{J + 1})$.

\begin{lemma}\label{lem:log2reg}
	Let $u$ be a minimizer of $G$ on $B_1$, with $G$ satisfying \eqref{eq:loglip}. Then for any $x,y\in B_{1/4}$, we have
	\[
	|\n u(x) - \n u(y)| \leq C(n, \eta, \|u\|_{L^2(B_1)}) \eta_1(|x - y|).
	\]
\end{lemma}

The $\eta$ dependence here is both in the form of $\eta_1$ (which depends on $\eta$ explicitly) and in the constant (which depends on the constant in Lemma \ref{lem:log2decay}).

\begin{proof}
	First, from applying Propositions \ref{prop:degiorgi} and \ref{prop:C1a}, we have that for $x\in B_{1/2}$,
	\begin{equation}\label{eq:log2reglip}
	|\n u(x)| \leq C.
	\end{equation}
	In particular, there is a fixed $r_0>0$ such that $\osc_{B_r(x)} u \leq c_*$ for every $x\in B_{1/4}$ and $r\leq r_0$. Applying Lemma \ref{lem:log2decay} with $r = r_0$, we see that for any $\r < r_0<\frac{1}{4}$,
	\[
	\fint_{B_\r(x)}|\n u - \fint_{B_{\r}(x)} \n u|^2 \leq C(r_0, \|u\|_{H^1(B_1)}) \frac{\eta^2(\r^2)}{\r^2} (1 + |\log \r|^2).
	\]
	This gives
	\begin{align*}
	|\fint_{B_{2^{-k-1}r_0}(x)}\n u - \fint_{B_{2^{-k}r_0}(x)}\n u|^2 & \leq \fint_{B_{2^{-k-1}r_0}(x)} |\n u - \fint_{B_{2^{-k}r_0}(x)}\n u|^2\\
	& \leq 2^n\fint_{B_{2^{-k}r_0}(x)} |\n u - \fint_{B_{2^{-k}r_0}(x)}\n u|^2\\
	& \leq C \frac{\eta^2(2^{-2k}r_0^2)}{2^{-2k}r_0^2} (1 + k^2)
	\end{align*}
	for $k \geq 0$. Summing, we have that the averages $\fint_{B_{2^{-k}r_0(x)}} \n u$ form a Cauchy series and
	\[
	\left|\n u(x) - \fint_{B_{2^{-k}r_0}(x)} \n u\right|^2 \leq C \eta_1^2(2^{-k}r_0).
	\]
	
	Now take any $x,y\in B_{1/4}$. If $|x - y|\geq r_0/4$, then \eqref{eq:log2reglip} directly implies the conclusion. If not, let $k$ be such that $2^{-k-1}r_0 < |x-y| < 2^{-k}r_0$; we then have
	\begin{align*}
	|\n u(x) - \n u(y)|^2 &\leq C \left|\fint_{B_{2^{1-k}r_0}(x)}\n u - \fint_{B_{2^{-k}r_0}(y)}\n u\right|^2 + C \eta_1^2(2^{-k}r_0)\\
	& \leq C\fint_{B_{2^{-k}r_0}(y)}\left|\n u - \fint_{B_{2^{1-k}r_0}(x)} \n u\right|^2 + C \eta_1^2(2^{-k}r_0)\\
	& \leq C\fint_{B_{2^{1-k}r_0}(x)}\left|\n u - \fint_{B_{2^{1-k}r_0}(x)} \n u\right|^2 + C \eta_1^2(2^{-k}r_0)\\
	&\leq C \eta_1^2(2^{-k}r_0)\\
	&\leq C \eta_1^2(|x - y|).
	\end{align*}
	This gives the conclusion.
\end{proof}

\section{Optimal Growth via the Weiss Formula} \label{sec:growth}

In this section we establish growth and monotonicity results for $u$ near points where $|u(0)|, |\n u(0)|$ are small. The results here are already interesting if $u(0) = |\n u(0)| = 0$ (i.e. at one-phase and branch points), though the greater generality will be helpful in the next section.

\begin{remark} \label{rem:modulus}
	Let $u$ be an $E_r$ minimizer on $B_1$ with $|u(0)| \leq \b$ and $ |\n u(0)| \leq \e$. Then applying Lemma \ref{lem:log2reg}, we see that $u$ admits the suboptimal modulus $\w(t) = Ct^2 (1 + |\log t|)^2$ for $t |\log t| \geq \e$ and $t^2 |\log t| \geq \b$. In other words, $\sup_{B_r} |u| \leq \w(r)$ as long as $r |\log r| \geq \e$, $r^2 |\log r| \geq \b$, and $r \leq \frac{1}{2}$, and $\w$ satisfies
	\[
	\int_0^1 \frac{\log(\w(r)) - 2 \log r}{r (1- 2\log r)^2} < \8.
	\]
	If $\e, \b = 0$, i.e. $u(0) = |\n u(0)| = 0$, then this is valid for all $r$. This integrability property is the only aspect of $\w$ which will be relevant below; note that it would also remain valid for $\w(t) = t (1 + |\log t|)^p$ for any $p$, though not for $\w(t) = t^{\a}$ with $\a < 2$ (hence the importance of the preceding section). Note that $\w$ here depends on $n$ and $\int_{B_1}u^2$ only.
\end{remark}

\begin{remark} \label{rem:Fcontrol}
	So long as $\sup_{B_r}|u| \leq \w(r)$ and $r \leq r_0(n, \w)$, we have that $F_r(u_r) \geq c(n) |u_r| \geq 0$ on $B_1$. Indeed,
	\begin{align*}
	F_r(u_r) &\geq c |u_r|(1 - \frac{\log |u_r|}{1 - 2 \log r})\\
	& \geq c |u_r|(1 - \frac{\log (\w(r)/\mu(r))}{1 - 2 \log r})\\
	& \geq c |u_r|(1 - \frac{\log (C(1 - \log r))}{1 - 2 \log r})\\
	& \geq c |u_r|.
	\end{align*}
\end{remark}

Let
\[
W(r) = \a(r)\int_{B_1} |\n u_r|^2 + 2 F_r(u_r) -2 \int_{\p B_1} u_r^2
\]
be a renormalized Weiss-type energy centered about the origin, where
\[
\a(r) = 1 - \frac{1}{2 \log r} \geq 1
\]
is an increasing function. Set
\[
\m(r) = r^2 (1 - 2\log r).
\]

Let us compute the derivative, in $r$, of this quantity $W$. First,
\[
\p_r u_r(x) = \p_r \frac{u(r x)}{\m(r)} = \frac{\n u_r(x) \cdot x}{r} - u_r(x) \frac{\m'(r)}{\m(r)}.
\]
The last factor can be written as
\[
\frac{\m'(r)}{\m(r)} = \frac{2}{r} (1 - \frac{1}{1 - 2 \log r}).
\]
We also have
\[
\p_r |\n u_r|^2 = 2 \n u_r \n (\p_r u_r),
\]
and
\[
\p_r F_r(u_r) = (\p_r F_r)(u_r) + f_r(u_r) (\p_r u_r), 
\]
where $f_r(t) = \p_t F_r(t) = \Delta u_r$. Combining and using the divergence theorem,
\begin{align*}
W'(r) &= \a(r)\int_{B_1} 2 \n u_r \n(\p_r u_r) + 2 (\Delta u_r) (\p_r u_r) + 2(\p_r F_r)(u_r) \\
&\qquad - 4 \int_{\p B_1}  u_r \p_r u_r + \a'(r)\int_{B_r} |\n u_r|^2 + 2 F_r(u_r) \\
& \geq 2\a(r) \int_{B_1} (\p_r F_r)(u_r) + 2 \int_{\p B_1} (\a(r)\n u_r \cdot x - 2 u_r) \p_r u_r \\
& \geq 2\a(r) \int_{B_1} (\p_r F_r)(u_r) + \frac{2 \a(r)}{r} \int_{\p B_1} (\n u_r \cdot x - \frac{2}{\a(r)} u_r)^2 \\
& := \frac{2 \a(r)}{r} \int_{\p B_1} (\n u_r \cdot x -  \frac{2}{\a(r)} u_r)^2 + Q(r).
\end{align*}
The first step used that $\a'$ is nonnegative, as is the integral that it is multiplied by, while the second step used that $\a(r) = (1 - \frac{1}{1 - 2 \log r})^{-1}$ and the computation of $\p_r u_r$.

A central point in our further discussions will be control over the error term $Q$. Let us expand it out:
\[
(\p_r F_r)(t) = (\l_+t_+ + \l_- t_-) \frac{2 r^3}{\mu^2(r)} (-\log |t| +  1 - \log(1 - 2 \log r)) .
\]
The key observation here is that since we are only concerned about bounding this from below, the single problematic situation in the above is when $u_r$ is large and hence $- |u_r| \log |u_r|$ is very negative. This motivates the following computation: (assuming $r < r_0$ below, so that $\a(r) \leq C$)
\begin{align*}
Q(r) &= 2\a(r) \int_{B_1} (\p_r F_r)(u_r)\\
&\geq - C \frac{r^3}{\mu^2(r)} \int_{B_1} |u_r|( (\log |u_r|)_+  + \log(1 - 2 \log r))\\
&\geq - C \frac{r^3}{\mu^2(r)} \int_{B_1} |u_r|( \log (\w(r)/\mu(r))  + \log(1 - 2 \log r))\\
& \geq - C \nu(r) \int_{B_1} |u_r| \\
& \geq - C \nu(r) \int_{B_1}F_r(u_r),
\end{align*}
where $\nu(r) := \frac{1 - \log(1 - 2 \log r) - \log(\w(r)/\mu(r)) }{r (1 - 2\log r)^2}$ is an integrable function on $[0,1]$. We used here that $|u_r| \leq \w(r)/\mu(t)$; the final inequality comes from Remark \ref{rem:Fcontrol}. To summarize, we have shown
\begin{equation} \label{eq:Qest}
Q(r) \geq - C \nu(r) \int_{B_1}F_r(u_r) \qquad r < r_0,
\end{equation}
with $\nu$ an integrable function.

For any $H^1$ function $u$, let $P u$ denote the quadratic harmonic polynomial on $B_1$ minimizing
\[
\int_{\p B_1} |u - P u|^2.
\]
One may check that
\begin{equation}\label{eq:subpoly}
\int_{B_1} |\n (u - Pu)|^2 - 2 \int_{\p B_1} |u - Pu|^2 = \int_{B_1} |\n u|^2 - 2 \int_{\p B_1} u^2
\end{equation}
by integrating by parts and using that $Pu$ is homogeneous of degree 2.

\begin{lemma}\label{lem:BMOest} Let $u$ be an $E_r$ minimizer on $B_1$ with $|u(0)| \leq 1$ and $|\n u(0)|\leq 1$. Then
	\[
	\int_{\p B_1} |u - P u|^2 \leq C [W_0(u; r) + 1].
	\]
\end{lemma}

Here $W_0(u; r) = \int_{B_1} |\n u|^2 + 2 F_r(u) - 2 \int_{\p B_1} u^2$.

\begin{proof}
	We argue by contradiction. Assuming this is not the case, there is a sequence of numbers $r_k \rightarrow r_\8 \in [0,1]$ and $u_k$ being $E_{r_k}$ minimizers such that
	\[
	\int_{\p B_1} |u_k - P u_k|^2 = C_k [W_0(u_k) + 1] = M_k,
	\]
	with $C_k \rightarrow \8$. Note that $M_k \rightarrow \8$ as well. Set $v_k = \frac{u_k - P u_k}{\sqrt{M_k}}$; these functions have
	\begin{align*}
	\int_{B_1} |\n v_k|^2 - 2\int_{\p B_1} v_k^2 & = M_k^{-1}[\int_{B_1} |\n (u_k - P u_k)|^2 - 2\int_{\p B_1} (u_k - P u_k)^2]  \\
	& = M_k^{-1}[\int_{B_1} |\n u_k|^2 - 2\int_{\p B_1} u_k^2]\\
	& \leq M_k^{-1}W_0(u_k; r_k) = \frac{W_0(u_k; r_k)}{C_k(W_0(u_k; r_k) + 1)} \leq C_k^{-1} \rightarrow 0.
	\end{align*}
	This gives
	\[
	\int_{B_1} |\n v_k|^2 \leq 2 + C_k^{-1} \leq 3,
	\]
	so passing to a subsequence, the $v_k$ converge weakly in $H^{1}$ to a $v$ with
	\[
	\int_{B_1} |\n v|^2 \leq \liminf_{k} \int_{B_1} |\n v_k|^2 \leq 2.
	\] 
	We also have that $v_k$ converge strongly in $L^2(B_1)$ and $L^2(\p B_1)$, the latter giving
	\[
	\int_{\p B_1} v^2 = \lim_k \int_{\p B_1} v_k^2 = 1
	\]
	by the definitions of $M_k$ and $v_k$. From Propositions \ref{prop:degiorgi} and \ref{prop:C1a} applied to $v_k$, we have that $v_k$ converge locally on $B_1$ in $C^{1,\a}$ topology, so in particular $|v(0)| + |\n v(0)| \leq \lim M_k^{-1/2} = 0$. From Lemma \ref{lem:conv}, we have that as $F_{r_k}(\sqrt{M_k} t)/M_k \rightarrow 0$ locally uniformly (in $t$), $v$ is harmonic on $B_1$. 
	
	It follows from the monotonicity of Almgren's frequency that
	\[
	2 \leq \lim_{s \rightarrow 0} \frac{s \int_{B_s} |\n v|^2}{\int_{\p B_s} v^2} \leq \frac{\int_{B_1} |\n v|^2}{\int_{\p B_1} v^2} \leq 2,
	\]
	and this equality implies that $v$ is a quadratic harmonic polynomial. However, $v_k$ is orthogonal to quadratic harmonic polynomials, and this passes to the limit by the strong convergence of $v_k$ in $L^2(\p B_1)$:
	\[
	0 = \lim_k \int_{\p B_1} v_k P v_k = \int_{\p B_1} v P v = \int_{\p B_1} v^2 
	\] 
	giving $v = 0$; this contradicts that $\int_{\p B_1} v_k^2 = 1$. 
\end{proof}

\begin{theorem}\label{thm:logreg} Let $u$ be an $E$ minimizer on $B_1$. There exists a $\r_W \leq 1$ and a $C_W$, depending only on $n, \l_\pm,$ and $\int_{B_1}u^2$, such that for any $M>0$ and $\r < \r_0 \leq \r_W(n)$, if $|\n u(0)|\leq \frac{\r}{2} |\log \r|$, $|u(0)|\leq \frac{\r^2}{4} |\log \r|$,
	\[
	E_{\r_0}(u_{\r_0}; B_1) \leq M,
	\]
	and
	\[
	\sup_{r \in [\r,\r_0]} \int_{B_1}F_r(u_r) \leq C_W(1 + M),
	\]
	then
	\[
	\sup_{r \in [\r/2,\r_0]} \int_{B_1} |\n (u_r - P u_r)|^2 + \int_{B_1}F_r(u_r) \leq C_W(1 + M).
	\]
\end{theorem}

\begin{proof}
	Observe that $|\n u_r(0)| \leq 1$, $|u_r(0)|\leq 1$, and $\sup_{B_r} |u|\leq \w(r)$ for $r \geq \r/2$ from Remark \ref{rem:modulus} and scaling.

	First, we claim that
	\[
	W(r) \leq W(\r_0) - \int_{r}^{\r_0} Q(s)ds \leq 2 M + C_1 C_W (1 + M) \int_{0}^{\r_W} \nu(s)ds
	\]
	for all $r \geq \r/2$. This is clear if $r\geq \r$ from \eqref{eq:Qest}, forcing $\r_W \leq r_0$. We can then use that $\int_{B_1} F_{r/2}(u_{r/2}) \leq C(n)\int_{B_1} F_{r}(u_{r})$, which is immediate from changing variables.

	Now,
	\begin{align*}
	\int_{B_1} &|\n (u_r - P u_r)|^2 + 2\a(r)\int_{B_1} F_r(u_r) \\
	&= \int_{B_1} |\n (u_r - P u_r)|^2 +  W(r) - \a(r)\int_{B_1} |\n u_r|^2 + 2 \int_{\p B_1} u_r^2 \\
	& \leq \int_{B_1} |\n (u_r - P u_r)|^2 +  W(r) - \int_{B_1} |\n u_r|^2 + 2 \int_{\p B_1} u_r^2 \\
	& = \int_{B_1} |\n (u_r - P u_r)|^2 +  W(r) - \int_{B_1} |\n (u_r - P u_r)|^2 + 2 \int_{\p B_1} (u_r - P u_r)^2 \\
	& \leq W(r) + 2 \int_{\p B_1} (u_r - P u_r)^2 \\
	& \leq W(r) + 2 C_2 (1 + W(r))\\
	& \leq 2 C_2 + (2 C_2 +1) (2M + C_1 C_W (1 + M) \int_{0}^{\r_W} \nu(s)ds)
	\end{align*}
	All the numbered constants depend only on $n$; the second line used that $\a(r)\geq 1$, after which we used \eqref{eq:subpoly}, our estimate of $W$, and Lemma \ref{lem:BMOest}. Therefore for $r \in [\r/2, \r]$,
	\begin{equation*}
	\int_{B_1} |\n (u_r - P u_r)|^2 + \int_{B_1} F_r(u_r) \leq C_3 (1 + M) +  C_4 C_W (1 + M) \int_{0}^{\r_W} \nu(s)ds.
	\end{equation*}
	We select $\r_W$ so that
	\[
	C_4 \int_{0}^{\r_W} \nu(s)ds \leq \frac{1}{2};
	\]
	this depends only on $n$. Then take $C_W$ so large that $C_3 \leq \frac{1}{2} C_W$; this gives
	\[
	\int_{B_1} |\n (u_r - P u_r)|^2 + F_r(u_r) \leq C_W (1 + M)
	\]
	as promised.
\end{proof}

\begin{corollary}\label{cor:optimalgrowth}
	Let $u$ be an $E$ minimizer on $B_1$ with $|u(0)| \leq r_1^2 |\log r_1|$ and $ |\n u(0)| \leq r_1 |\log r_1|$. Then there is a $C = C(n, \l_\pm, \|u\|_{L^2(B_1)})$ such that
	\[
	\sup_{B_r}|u| \leq C r^2 (1 + |\log r|).
	\]
	for $r_1 \leq r< \frac{1}{2}$.
	
	In particular, if $u(0) = |\n u(0)| = 0$, this holds for all $r$.
\end{corollary}

\begin{proof}
	Apply Remark \ref{rem:modulus} to $u$ to deduce that 
	\begin{equation}\label{eq:logoff}
	\sup_{B_r} |u| \leq \w(r)
	\end{equation}
	for $r \in [r_1, \frac{1}{2}]$. From Remark \ref{rem:Fcontrol}, this gives $F_r(u_r) \geq c |u_r|$ for $r_1 \leq r \leq r_0$. Set $\r_0 = \min\{r_0, \r_W \}$ and apply Theorem \ref{thm:logreg} repeatedly (with $M$ set to $E_{r_0}(u_{r_0}; B_1) \geq \int_{B_1}F_{r_1}(u_{r_1})$) to get
	\begin{equation}\label{eq:fullgrowth}
	\sup_{r \in [r_1,\r_0]} \int_{B_1} |\n (u_r - P u_r)|^2 + F_r(u_r) \leq C_W (1 + M).
	\end{equation}
	
	Next, we consider $P u_r$. Select an orthonormal (in $L^2(\p B_1)$) basis for the quadratic harmonic polynomials on $B_1$, $\{Q_i\}_{i = 1}^{J}$ (the space spanned by them is isomorphic to that of trace-free symmetric matrices over $\R^n$, via $Q \mapsto D^2 Q(0)$, so $J = n(n+1)/2 -1$). Let $S = \max_{B_1, i} |Q_i|$. Then
	\[
	\int_{B_1}|\n P u_r|^2 = 2 \int_{\p B_1} |P u_r|^2 = 2 \sum_{i = 1}^J (\int_{\p B_r} Q_i u_r)^2 \leq 2 J S^2 ( \int_{\p B_1} |u_r| )^2.
	\]
	By applying Chebyshev's inequality, we have that for some $\r \in [\frac{1}{2}, 1]$,
	\[
	\int_{B_1}|\n P u_{r\r}|^2 \leq C ( \int_{\p B_1} |u_{r\r}| )^2 \leq C (\int_{\p B_{\r}} |u_r|)^2 \leq C (\int_{B_1} |u_r|)^2.
	\]
	As $F_r(u_r) \geq c |u_r|$, this may be rewritten:
	\[
	\int_{B_1}|\n P u_{r\r}|^2 \leq C (\int_{B_1} F_r(u_r))^2 \leq C (1 + M)^2.
	\]
	Combining this with the estimate on $u_{s} - Pu_s$ from \eqref{eq:fullgrowth},
	\[
	\int_{B_1}|\n u_{r/2}|^2 \leq C\int_{B_{\frac{1}{2\r}}}|\n u_{r\r}|^2 \leq C \int_{B_1}|\n u_{r\r}|^2 \leq C(M).
	\]
	This is valid for every $r_1 \leq r \leq r_0$, regardless of the choice of $\r$ previously.
	
	Finally, we may directly apply Proposition \ref{prop:degiorgi} to obtain that
	\[
	\sup_{B_{1/2}} |u_r| \leq C(\int_{B_1}|\n u_r|^2) \leq C(M)
	\]
	for all $r_1/2 \leq r \leq r_0/2$. Rescaling this gives the conclusion so long as $r \leq r_0/2$, but for $r \geq r_0/2$ the conclusion is immediate from \eqref{eq:logoff}.
\end{proof}

\section{Optimal Regularity} \label{sec:proofmain}

Corollary \ref{cor:optimalgrowth} provides an optimal growth control for a minimizer $u$ near points where $u(0) = |\n u(0)| = 0$, and a useful estimate when $|\n u(0)|$ is small. In order to turn this into a regularity statement, we must consider the opposite situation: $u(0) = 0$ but $\n u(0)$ is large. The key point here is that in this setting, the behavior in directions orthogonal to $\n u$ is extremely regular, so the problem is largely one-dimensional. To exploit this we use a change of variables argument.

\begin{lemma}\label{lem:biggrad}
	Let $u$ be a minimizer of $E_r$ on $B_1$ with $u(0) = 0$, and assume that $|\n u(0)| \geq \frac{1}{4}$. Then
	\[
		|\n u(x) - \n u(y)| \leq C | x - y|(1 + \frac{\log |x - y|}{\log r})
	\]
	for $x, y \in B_{r_1}$, where $C, r_1$ depend only on $n$, $\l_\pm$, and $\int_{B_{1}}u^2$.
\end{lemma}

\begin{proof}
	First, apply Lemma \ref{lem:log2reg} to $u$ to obtain that
	\[
		|\n u(x) - \n u(y)| \leq C | x - y||\log |x - y||^2
	\]
	for $x, y \in B_{1/2}$, and $|\n u(0)|\leq C$. We may therefore write
	\[
		|u(x) - \n u(0) x|\leq \w(|x|),
	\]
	where $\w$ is as in Remark \ref{rem:modulus}. Select a coordinate system $(x', x_n)$ with $e_n = \frac{\n u(0)}{|\n u(0)|}$, and let $Q_s = \{(x', x_n) : |x'|\leq s, |x_n|\leq s  \}$ be a cylinder.

	For $s$ small and fixed, we have that on $Q_s$ the map $t \mapsto u(x + t e_n)$ is strictly monotone, and so the change of variables
	\[
		\psi(x', x_n) = (x', u(x', x_n)) : Q_s \rightarrow U
	\]
	is a $C^{1,\a}$ diffeomorphism. Set $v : U \rightarrow \R$ to be the $n$-th component of the inverse map $\psi^{-1} : U \rightarrow Q_s$. Direct computation gives the relations
	\[
		v_n \circ \psi = \frac{1}{u_n} \qquad v_i \circ \psi = - \frac{u_i}{u_n},
	\]
	where $i < n$ and subscripts denote derivatives. A further computation gives
	\[
		(\Delta u) \circ \psi^{-1} = L v := - \frac{1}{v_n} \sum_{i < n} v_{ii} + \frac{2}{v_n^2} \sum_{i < n} v_i v_{n i} - \frac{v_{nn}}{v_n^3} (1 + \sum_{i < n} v_i^2).
	\]
	These computations may be found in \cite{KN}.
	
	Now, $|u_n - u_n(0)|\leq \frac{1}{8}$ for small enough $s$, while $u_n(0) = |\n u(0)| \in [\frac{1}{4}, C]$, so $u_n, v_n \in [c, C]$ on $Q_s, U$, respectively. On the other hand, $|u_i| = |u_i - u_i(0)|\leq C s |\log s|^2$, so $|v_i|$ is small in terms of $s$. Thus for small enough $s$, $L$ is elliptic on $U$, with ellipticity constant $C$. Hence on $U$, $v$ satisfies the elliptic equation
	\[
		L v = f_r(u \circ \psi^{-1}) = f_r(y_n).
	\]
	We also have $|v|\leq C$ on $U$. From standard elliptic theory and the fact that $|f_r(y_n)|\leq C( 1 + |\frac{\log |y_n|}{\log r}|)$, we may obtain that $\|v_{ij}\|_{L^p}(Q_{s_0}) \leq C$ for some fixed $Q_{s_0} \cc U$ and $p > n$.

	Apply the partial Schauder estimate from Theorem 2.10 in \cite{DK} to $v$ on $Q_{s_0}$ to obtain that
	\[
		[v_{ij}]_{C^{0,\a}(Q_{s_0/2})} \leq C( \sup_{y'_1 \neq y'_2} \frac{|f_r(t) - f_r(t)|}{|y'_1 - y'_2|^\a} + C(\|\n v\|_{C^{0,\a}}) \|D^2 v\|_{L^p}(Q_{s_0}) ) \leq C
	\]
	for any $\a < 1$, $i \leq n$, and $j < n$. While $f_r(y_n)$ was required to be continuous there, the estimate is independent of its boundedness or continuity, and so can be obtained in the fashion written here by an approximation argument. As, for example, $|D^2 u(0, s_0/4)| \leq C$ from local elliptic estimates, this can used to bound the full $C^{0,\a}$ norm of these derivatives:
	\[
		\|v_{ij}\|_{C^{0,\a}(Q_{s_0/2})} \leq C.
	\]
	As a consequence of this, we may rewrite $L v$ as
	\[
		f_r(y_n) = L v (y) = h_1(y) - h_2(y) v_{nn}(y),
	\]
	where $h_1, h_2$ are $C^{0, \a}$ functions of $y$ with $h_2 \geq 1$. From this, we immediately obtain that 
	\[
		|v_{nn}(y)|\leq C( 1 + \frac{|\log |y_n||}{|\log r|}),
	\]
	and so
	\[
		|v_i(x) - v_i(y)|\leq C|x - y|(1 + \frac{|\log |x - y||}{|\log r|})
	\]
	by integrating. Changing variables back, we learn that on $Q_{s_1} \ss \phi^{-1}(Q_{s_0})$,
	\[
		|\n u(x) - \n u(y)|\leq C|x - y|(1 + \frac{|\log |x - y||}{|\log r|}).
	\]
\end{proof}

\begin{theorem}
	Let $u$ be minimizer of $E$ on $B_1$. Then for $x \in B_{1/2}$,
	\[
		|\n u(x) - \n u(0)|\leq C |x| (1 + |\log |x||),
	\]
	where $C = C(n, \l_\pm, \int_{B_{1}}u^2)$.
\end{theorem}

\begin{proof}
	Applying Lemma \ref{lem:log2reg}, we know that
	\[
		|\n u(x) - \n u(0)| \leq C |x|(1 + |\log |x|)^2
	\]
	on $B_{1/2}$. This implies the conclusion for $|x| \geq \frac{1}{8}$, so we only need to consider $|x|\leq \frac{1}{8}$ below.

	We first consider the case of $u(0) = 0$. Set $s\in [0, \frac{1}{2}]$ to be the smallest value of $r$ for which 
	\[
		|\n u_r(0)| = \frac{1}{r(1 - 2\log r)} |\n u(0)| \leq \frac{1}{4}.
	\]
	Then for $r \in [s, 1/2]$, we may apply Corollary \ref{cor:optimalgrowth} to obtain that
	\[
		\sup_{B_1} |u_r| \leq C.
	\]
	Applying Propositions \ref{prop:degiorgi} and \ref{prop:C1a} to $u_r$ gives
	\[
		\sup_{B_{1/2}} |\n u_r| \leq C.
	\]
	If $|x|\geq \frac{s}{2}$, then using $r = 2|x|$ gives
	\begin{align*}
		|\n u(x) - \n u(0)| &\leq |\n u(x)| + |\n u(0)| \\
		&\leq r(1 - 2\log r)|\n u_r(x/r)| + \frac{1}{4}r(1 - 2\log r)\\
		& \leq Cr(1 - 2\log r)\\
		& \leq C |x|(1 + |\log |x||).
	\end{align*}
	On the other hand, if $|x| \leq \frac{s}{2}$, we apply Lemma \ref{lem:biggrad} to $u_s$, to get
	\begin{align*}
		|\n u(x) - \n u(0)| &\leq s(1 - 2\log s)|\n u_s(x/s) - \n u_s(0)|\\
		& \leq Cs|\log s| \frac{|x|}{s} (1 + \frac{|\log\frac{|x|}{s}|}{|\log s|}) \\
		&\leq C |x|(1 + |\log |x||).
	\end{align*}
	
	Now we turn to the case of $u(0) \neq 0$. We proceed similarly to the above, except now set $s\in [0, \frac{1}{2}]$ to be the smallest value of $r$ for which 
	\begin{equation}\label{eq:final1}
	|\n u_r(0)| \leq \frac{1}{2} \qquad |u_r(0)| \leq \e,
	\end{equation}
	with $\e$ to be chosen.	For $r \in [s, 1/2]$, we may apply Corollary \ref{cor:optimalgrowth} to obtain
	\[
		\|\n u_r\|_{C^{0,\a}(B_{1/2})} \leq C
	\]
	as before, and this implies the conclusion so long as $|x|\geq c_0 s$ for any fixed $c_0$, also to be chosen. At this point there are two cases to consider, depending on which of the two criteria in \eqref{eq:final1} failed to be satisfied first.
	
	Consider first the case that $|\n u_s(0)| = \frac{1}{2}$ and $|u_s(0)| \leq \e$. We know that $\|\n u_s\|_{C^{0,\a}(B_{1/2})} \leq C$, so
	\[
		|u_s(y) - y \cdot \n u_s(0)| \leq |u(0)| + C|y|^{1 + \a} \leq \e + C|y|^{1 + \a}.
	\]
	Set $y = \pm t \frac{\n u_s(0)}{|\n u_s(0)|}$; then
	\[
		|u_s(y) \pm \frac{1}{2} t| \leq \e + Ct^{1+\a} \leq \e + \frac{1}{8}t
	\]
	if we fix $ t$ with $8 Ct^\a \leq 1$ and $t \leq \frac{1}{16}$. Choose $\e = \frac{t}{8}$; then $u_s(y)$ has opposite sign at the two values of $t$, and so there must be a point $z \in B_{t}$ with $u_s(z) = 0$. From the same estimates, we have that $|\n u(z)| \geq \frac{1}{2} - Ct^\a \geq \frac{1}{4}$. Apply Lemma \ref{lem:biggrad} to $u_s$ on $B_{1}(z)$ to obtain that
	\[
		|\n u_s(x) - \n u_s(y)| \leq C |x - y|(1 + \frac{|\log |x - y||}{|\log s|})
	\]
	for $x, y \in B_{1/2}(z)$. In particular this holds with $y = 0$ and $ x\in B_{1/4} \ss B_{1/2}(z)$, which leads to
	\[
		|\n u(x) - \n u(0)| \leq C |x| (1 + |\log |x||)
	\]
	for $ x\in B_{s/4}$ after rescaling.
	
	Now for the opposite case: $u_s(0) = \e$, with $\e$ chosen above (changing sign if negative, without loss of generality) while $|\n u_s(0)|\leq \frac{1}{2}$. We still have that $\|\n u_s\|_{C^{0,\a}(B_{1/2})} \leq C$, so on $B_{c \e}$, this implies $u_s \geq \frac{\e}{2}$. Using only that $u_s \in [\frac{\e}{2}, C]$, $|\n u_s|\leq C$, and the PDE $-\Delta u_s = f_s(u_s)$ allows us to estimate directly that
	\[
		|\Delta u_s| \leq C \qquad |\n \Delta u_s| \leq C.
	\]
	From standard Schauder estimates,
	\[
		|\n u_s(x) - \n u_s(0)|\leq C |x|
	\]
	for $x \in B_{c \e/2}$, and rescaling this gives
	\[
		|\n u(x) - \n u(0)| = s (1 - 2 \log s) |\n u_s(x/s) - \n u_s(0)| \leq C s (1 - 2 \log s) \frac{|x|}{s} \leq C |x| (1 + |\log |x||)
	\]
	for $|x|\leq cs\e/2$.
	
	Set $c_0 = \min\{ c \e/2, \frac{1}{4}\}$ above to obtain the conclusion.
\end{proof}

\section{Nondegeneracy} \label{sec:nondegeneracy}

Unlike the maximal growth estimate of the previous section, showing a minimal growth rate for the solution away from free boundary points can be done with elementary modifications of standard arguments (see e.g. \cite{PSU} for the classical case).

\begin{lemma}\label{lem:lb}
	Let $u$ be an $E$ minimizer on $B_1$ with $u(0) = 0$. Then there is a $c = c(n, \l, \|u\|_{H^1(B_1)})$ such that, for $r \leq \frac{1}{2}$, either
	\[
	\sup_{B_r} u^+ \geq c r^2 (1 + |\log r|)
	\]
	or $\{u > 0\}\cap B_{r/2} = \emptyset$, and either
	\[
	\sup_{B_r} u^- \geq c r^2 (1 + |\log r|)	
	\]
	or $\{u < 0\}\cap B_{r/2} = \emptyset$.
\end{lemma}

\begin{proof}
	We note that on $B_r$, $|u|\leq C r$ from Propositions \ref{prop:degiorgi} and \ref{prop:C1a}. This gives that, on $\{u>0\}$,
	\[
	\Delta u = - \l_+ \log u \geq  - \frac{1}{2}\l_+ \log r
	\]
	for $r < r_0(C)$ small enough. Then select a point $x_0 \in B_{r/2}\cap \{u>0\}$; if there is no such point, the conclusion follows directly. If there is, $v = u - \frac{\l_+|\log r|}{4n}|x - x_0|^2$ is subharmonic on $\{u > 0\}\cap B_r$. At $x_0$, $v>0$, while on $\p \{u > 0\} \cap B_r$, $v < 0$; it follows from the maximum principle that there must be a point $y$ in $\p B_r$ with $v(y) \geq 0$. This gives
	\[
	u(y) \geq \frac{\l_+}{16n} |\log r| r^2,
	\]
	implying the conclusion.
	
	If, on the other hand, $r > r_0$, one may instead use the point found for $r = r_0$ above, adjusting the constant by a factor depending on $r_0$ only to obtain the conclusion.
\end{proof}

\section*{Acknowledgments}

DK was partially supported by the NSF MSPRF fellowship DMS-1502852. Much of this work was conducted during his visit to the KTH Royal Institute of Technology, and he is grateful to it for making this possible. HS was supported by the Swedish Research Council

\bibliographystyle{plain}
\bibliography{log_obstacle_fb}

\begin{thebibliography}{1}

\bibitem{ASUW}
John Andersson, Henrik Shahgholian, Nina~N. Uraltseva, and Georg~S. Weiss.
\newblock Equilibrium points of a singular cooperative system with free
  boundary.
\newblock {\em Advances in Mathematics}, 280:743 -- 771, 2015.

\bibitem{DK}
Hongjie Dong and Seick Kim.
\newblock {Partial Schauder Estimates for Second-Order Elliptic and Parabolic
  Equations: A Revisit}.
\newblock {\em International Mathematics Research Notices}, 2019(7):2085--2136,
  08 2017.

\bibitem{GG2}
Mariano Giaquinta and Enrico Giusti.
\newblock On the regularity of the minima of variational integrals.
\newblock {\em Acta Math.}, 148:31--46, 1982.

\bibitem{GG}
Mariano Giaquinta and Enrico~and Giusti.
\newblock Differentiability of minima of nondifferentiable functionals.
\newblock {\em Invent. Math.}, 72(2):285--298, 1983.

\bibitem{KN}
David Kinderlehrer and Louis Nirenberg.
\newblock Regularity in free boundary problems.
\newblock {\em Annali della Scuola Normale Superiore di Pisa - Classe di
  Scienze}, Ser. 4, 4(2):373--391, 1977.

\bibitem{PSU}
Arshak Petrosyan, Henrik Shahgholian, and Nina Uraltseva.
\newblock {\em Regularity of free boundaries in obstacle-type problems}.
\newblock American Mathematical Society, Providence, Rhode Island, 2012.

\bibitem{QS}
Henrik Shahgholian and Olivaine de~Queiroz.
\newblock A free boundary problem with log--term singularity.
\newblock {\em Interfaces and Free Boundaries}, 19, 12 2017.

\end{thebibliography}

\end{document}